\documentclass[12pt]{article}
\usepackage[T1]{fontenc}
\usepackage[utf8]{inputenc}
\usepackage[a4paper,textwidth=13.5cm,textheight=23cm]{geometry}
\usepackage{graphicx}
\usepackage{amsmath,amssymb,amsthm,mathtools}
\usepackage{xcolor}
\usepackage{comment}
\usepackage{microtype}
\usepackage{tikz}
\usetikzlibrary{arrows.meta}
\usepackage[hidelinks]{hyperref}
\usepackage[nameinlink,noabbrev,capitalise]{cleveref}

\newtheorem{theorem}{Theorem}[section]
\newtheorem{proposition}[theorem]{Proposition}
\newtheorem{lemma}[theorem]{Lemma}
\newtheorem{corollary}[theorem]{Corollary}

\theoremstyle{definition}
\newtheorem{remark}[theorem]{Remark}
\newtheorem{definition}[theorem]{Definition}

\newcommand{\Q}{\mathbb Q}
\newcommand{\N}{\mathbb N}
\newcommand{\R}{\mathbb R}
\newcommand{\F}{\mathbb F}

\newcommand{\Gal}{\operatorname{Gal}}
\newcommand{\Disc}{\operatorname{Disc}}
\newcommand{\Res}{\operatorname{Res}}
\newcommand{\NP}{\operatorname{NP}}
\newcommand{\sign}{\operatorname{sign}}
\newcommand{\ac}{\operatorname{ac}}
\newcommand{\sqf}{\operatorname{sf}}
\newcommand{\ud}{\overline{\operatorname{d}}}

\newcommand{\vp}{v_p}

\newcommand{\abs}[1]{\left| #1 \right|}
\newcommand{\floor}[1]{\left\lfloor #1 \right\rfloor}
\newcommand{\Leg}[1]{\left(\frac{#1}{p}\right)}
\title{Galois Groups of Shifted Exponential Polynomials}
\author{Shachar Weinbaum\thanks{Undergraduate student, Tel Aviv University,
  Tel Aviv, Israel; research fellow, Technion---Israel Institute of Technology,
  Haifa, Israel. Email:
  \href{mailto:shacharwein@gmail.com}{\texttt{shacharwein@gmail.com}}}}
\date{September 29, 2026}

\begin{document}
\baselineskip=17pt

\maketitle

\begin{abstract}
Consider the shifted truncations of the exponential series and their Galois groups:
\[
  F_{n,a}(x) \coloneqq \sum_{j=0}^{n}\frac{x^j}{(a+j)!}\qquad G_{n,a}\coloneqq \Gal(F_{n,a}/\Q).
\]
In this paper we prove that for every fixed integer $a>0$, the set of $n$ for which $G_{n,a}=S_n$ has natural density one, thus answering a question of Katz and Rivin.
\end{abstract}

\noindent\textit{2020 Mathematics Subject Classification:}
Primary 11R32; Secondary 12E05.

\noindent\textit{Key words and phrases:}
Galois groups, truncated exponential series, irreducibility, Newton polygons,
discriminants, Dirichlet $L$-functions.

\section{Introduction and outline}
\label{sec:introduction}

Fix an integer $a>0$ and put
\begin{equation}
  F_{n,a}(x)=\sum_{j=0}^{n}\frac{x^j}{(a+j)!}\in\Q[x],
  \qquad
  P_{n,a}(x)=(n+a)!F_{n,a}(x)
  =\sum_{j=0}^{n}b_jx^j,
  \label{eq:family}
\end{equation}
where
\[
  b_j=\frac{(n+a)!}{(a+j)!}.
\]
The polynomial $P_{n,a}$ is monic and integral, and multiplication by the nonzero constant $(n+a)!$ does not change the splitting field. For $f\in K[x]$, we write $\Gal(f/K)$ for the Galois group of the splitting-field extension of $f$. Thus
\[
  G_{n,a}:=\Gal(F_{n,a}/\Q)=\Gal(P_{n,a}/\Q)\le S_n.
\]
Moreover,
\begin{equation}
  \Disc(P_{n,a})=((n+a)!)^{2n-2}\Disc(F_{n,a}),
  \label{eq:disc-scaling}
\end{equation}
so, if the polynomials are squarefree, the two discriminants define the same class in $\Q^\times/\Q^{\times2}$.

The unshifted case belongs to the classical theory initiated by Schur \cite{schur1929,schur1930} and revisited through Newton polygons by Coleman \cite{coleman1987}.  Shorey and Tijdeman \cite{shorey-tijdeman} developed broad irreducibility results for these polynomials.  More recently, Bary-Soroker and Ben-Porath \cite[Theorems~1--2, pp.~184--185]{bary-benporath} studied generalized Laguerre polynomials and trimmed exponentials, proving full symmetric Galois groups with probability tending to one in a different asymptotic regime.  A recent survey \cite{rabarison-pazuki-molin} places these results in the wider study of truncations of power series. Finally, Katz and Rivin \cite[\S\S1, 4, pp.~1--5]{katz-rivin} have presented general conjectures, based on much experimental evidence, about the Galois groups of truncations of many series of general interest.

Their experimental analysis suggested $G_{n,a} = S_n$ for virtually all $n$ and $a>0$. We establish this in terms of natural density:
\begin{theorem}
\label{thm:main-target}
For every fixed integer $a>0$, the set
\[
  \{n\ge1:G_{n,a}=S_n\}
\]
has natural density one.
\end{theorem}
In the notation of Katz and Rivin \cite[\S\S1, 4]{katz-rivin}, the degree-$n$ truncation of the $a$-shift $f_a(x)\coloneqq\sum_{i\ge0}x^i/(i+a)!$ of $\exp(x)$ is $F_{n,a}$, and $S(f_a)$, $A(f_a)$, $\mathcal N(f_a)$ denote the limiting proportions of degrees $n$ for which $G_{n,a}$ is $S_n$, $A_n$, or neither, respectively. Theorem~\ref{thm:main-target} states that these limits exist, with
\[
  S(f_a)=1,\qquad A(f_a)=\mathcal N(f_a)=0.
\]
In particular, it answers \cite[Questions~1.1 and~1.2]{katz-rivin} for every shift of the exponential series.
\begin{figure}[ht]
  \centering
  \resizebox{\textwidth}{!}{%
  \begin{tikzpicture}[
    source/.style={draw,rounded corners,align=center,text width=78mm,
                  minimum height=9mm,inner sep=2mm,fill=gray!10},
    branchhead/.style={draw,rounded corners,align=center,text width=58mm,
                      minimum height=8mm,inner sep=1.5mm,font=\small\bfseries},
    flowstep/.style={draw,rounded corners,align=center,text width=27mm,
                    minimum height=11mm,inner sep=1.5mm,font=\footnotesize},
    flowmini/.style={draw,rounded corners,align=center,text width=23mm,
                    minimum height=11mm,inner sep=1.5mm,font=\scriptsize},
    casebox/.style={draw,rounded corners,align=center,text width=20mm,
                   minimum height=10mm,inner sep=1.2mm,font=\scriptsize},
    flowwide/.style={draw,rounded corners,align=center,text width=59mm,
                    minimum height=10mm,inner sep=1.5mm,font=\footnotesize},
    flowoutcome/.style={draw,rounded corners,align=center,text width=59mm,
                       minimum height=11mm,inner sep=1.5mm,font=\footnotesize},
    conclusion/.style={draw,rounded corners,align=center,text width=88mm,
                      minimum height=12mm,inner sep=2mm,fill=green!10},
    arrow/.style={-{Latex[length=2mm]},thick},
    guide/.style={gray!55,densely dashed}
  ]
    \node[source] (family) at (0,0)
      {shifted family $F_{n,a}$ and $P_{n,a}$};

    \node[branchhead,fill=blue!10] (largehead) at (-3.65,-1.25)
      {Large-subgroup branch};
    \node[branchhead,fill=orange!12] (dischead) at (3.65,-1.25)
      {Discriminant branch};

    \draw[guide] (0,-1.75) -- (0,-11.05);

    \node[flowstep,fill=blue!5] (factor) at (-5.15,-2.65)
      {Shorey--Tijdeman};
    \node[flowstep,fill=blue!5] (newton) at (-2.15,-2.65)
      {Newton polygons};
    \node[flowstep,fill=blue!5] (irreducible) at (-5.15,-4.15)
      {density-one irreducibility};
    \node[flowstep,fill=blue!5] (cycles) at (-2.15,-4.15)
      {$q$- and $(p-a)$-cycles};
    \node[flowwide,fill=blue!7] (primitive) at (-3.65,-5.75)
      {primitivity};
    \node[flowwide,fill=blue!7] (jordan) at (-3.65,-7.15)
      {Jordan's theorem applied to the $q$-cycle};
    \node[flowoutcome,fill=blue!12] (alternating) at (-3.65,-8.55)
      {$A_n\subseteq G_{n,a}$ on a density-one set};

    \node[flowstep,fill=orange!6] (sign) at (2.15,-2.65)
      {sign of $\Disc(P_{n,a})$};
    \node[flowstep,fill=orange!6] (local) at (5.15,-2.65)
      {discriminant in $\Q_p$};
    \node[flowstep,fill=orange!6] (negative) at (2.15,-4.15)
      {$n\equiv2,3\pmod4$:\\negative discriminant};
    \node[flowstep,fill=orange!6] (localdata) at (5.15,-4.15)
      {$p$-adic valuation and angular component};
    \node[flowmini,fill=orange!6] (oddcase) at (2.15,-5.75)
      {odd $a$:\\odd discriminant valuation};
    \node[flowmini,fill=orange!6] (evencase) at (5.15,-5.75)
      {even $a$:\\angular component};
    \node[casebox,fill=orange!8] (onecase) at (3.7,-7.5)
      {$n\equiv1\pmod4$:\\quadratic reciprocity};
    \node[casebox,fill=orange!8] (zerocase) at (5.95,-7.5)
      {$n\equiv0\pmod4$:\\biased characters};
    \node[flowmini,fill=orange!8] (characters) at (5.95,-9.15)
      {zero-density bounds};
    \node[flowoutcome,fill=orange!13] (nonsquare) at (3.65,-10.8)
      {$\Disc(P_{n,a})\notin\Q^{\times2}$ on a density-one set};

    \node[conclusion] (symmetric) at (0,-12.45) {%
      $A_n\subseteq G_{n,a}$ and $\Disc(P_{n,a})\notin\Q^{\times2}$
      $\Longrightarrow$ $G_{n,a}=S_n$ on a density-one set};

    \draw[arrow] (family.south west) -- (largehead.north);
    \draw[arrow] (family.south east) -- (dischead.north);

    \draw[arrow] ([xshift=-18mm]largehead.south) -- (factor.north);
    \draw[arrow] ([xshift=18mm]largehead.south) -- (newton.north);
    \draw[arrow] (factor.south) -- (irreducible.north);
    \draw[arrow] (newton.south) -- (cycles.north);
    \draw[arrow] (irreducible.south) -- ([xshift=-18mm]primitive.north);
    \draw[arrow] (cycles.south) -- ([xshift=18mm]primitive.north);
    \draw[arrow] (primitive.south) -- (jordan.north);
    \draw[arrow] (jordan.south) -- (alternating.north);

    \draw[arrow] ([xshift=-18mm]dischead.south) -- (sign.north);
    \draw[arrow] ([xshift=18mm]dischead.south) -- (local.north);
    \draw[arrow] (sign.south) -- (negative.north);
    \draw[arrow] (local.south) -- (localdata.north);
    \draw[arrow] (localdata.south) -- (oddcase.north);
    \draw[arrow] (localdata.south) -- (evencase.north);
    \draw[arrow] (evencase.south) -- (onecase.north);
    \draw[arrow] (evencase.south) -- (zerocase.north);
    \draw[arrow] (zerocase.south) -- (characters.north);

    \draw[arrow] (negative.west) -- ++(-4mm,0) |- (nonsquare.west);
    \draw[arrow] (oddcase.south) -- ([xshift=-14mm]nonsquare.north);
    \draw[arrow] (onecase.south) -- (nonsquare.north);
    \draw[arrow] (characters.south) -- ([xshift=19mm]nonsquare.north);

    \draw[arrow] (alternating.south) -- (symmetric.north west);
    \draw[arrow] (nonsquare.south) -- (symmetric.north east);
  \end{tikzpicture}%
  }
  \caption{Logical flow of the proof. The left branch produces density-one containment of $A_n$ from irreducibility and explicit cycles. The right branch proves density-one nonsquareness of the discriminant by separating the sign and local cases.}
  \label{fig:flow}
\end{figure}
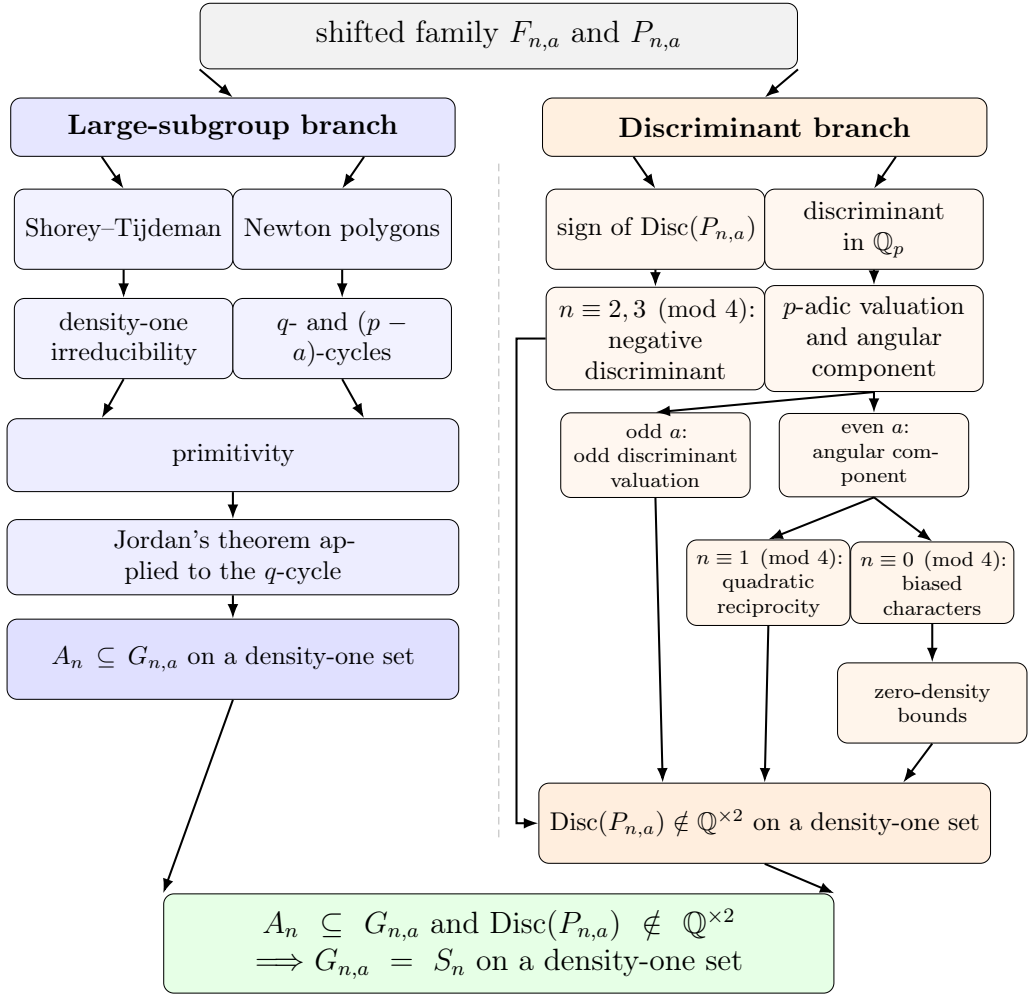

The structure of the paper is described in \cref{fig:flow}. Section~\ref{sec:irreducible} proves density-one irreducibility using work of Shorey and Tijdeman: Subsection~\ref{subsec:smoothness} converts the resulting smoothness bounds into degree bounds, and Subsection~\ref{subsec:small-factors} shows that the remaining small factors occur with density zero. 

Section~\ref{sec:primitive} proves $A_n\le G_{n,a}$ for sufficiently large $n$ for which $F_{n,a}$ is irreducible. First, in Subsections~\ref{subsec:q-cycle} and~\ref{subsec:p-a-cycle}, we establish the existence of two cycles in $G_{n,a}$. in Subsection~\ref{subsec:block-structure} we use these cycles to eliminate block structures, establishing primitivity of $G_{n,a}$. Jordan's theorem then gives $A_n \le G_{n,a}$ on a density-one set.

Section~\ref{sec:alternating-symmetric} proves that $\Disc(P_{n,a})$ is a nonsquare on a density-one set. This establishes \cref{thm:main-target}. In Subsection~\ref{subseq:disc-sign}, we show the discriminant is negative for $n\equiv2,3\pmod 4$. Subsection~\ref{subseq:disc-in-Qp} calculates the valuation and angular component (see Definition~\ref{def:ac}) of the discriminant in $\Q_p$. For odd $a$, the $p$-adic valuation is odd, resolving that case. In Subsection~\ref{subseq:disc-ac-Leg} the Legendre symbol of the discriminant's angular component is calculated and simplified. For even $a$, a quadratic reciprocity argument shows the angular component for $n\equiv1\pmod 4$ cannot be a square mod $p$ for big enough $n$. Subsection~\ref{subseq:disc-n0mod4} deals with the remaining class of $n\equiv0\pmod 4$. We show that a positive upper density of square discriminants would generate many biased quadratic characters, which contradicts a zero-density estimate for Dirichlet $L$-functions. Combining the $A_n$ containment with nonsquareness gives $G_{n,a}=S_n$ on a density-one set for every fixed $a>0$.

\subsection*{Acknowledgements}
This work was done during a two-month REU program at Princeton University, supported by both Tel Aviv University and Princeton University. The author would like to thank Prof. Nick Katz for suggesting the problem and for his guidance throughout the program.
\subsection*{Declaration of AI use}
The author has used ChatGPT 5.6 Sol and Claude Opus 5.5 during the research and writing of this work. The former has been used for generating code for numerical experiments, generating \LaTeX{} figures, searching for and summarizing relevant literature, developing the proofs of Lemma~\ref{lem:dense-bad-pairs} and Proposition~\ref{prop:little-intersections}, and helping to justify the simplification in the proof of Lemma~\ref{lem:ac-disc}. Both models have been used for review and proofreading of the argument and prose. The author takes full responsibility for the content and correctness of this paper.
\pagebreak
\section{Irreducibility with Density One}\label{sec:irreducible}
In this section we will prove the following statement:
\begin{theorem}\label{thm:density-irreducible}
  For every fixed integer $a>0$, the set
\[
  \{n\ge1:P_{n,a} \text{ is irreducible}\}
\]
has natural density one.
\end{theorem}
This is proven using three main ideas:

\begin{enumerate}
  \item If $P_{n,a}$ has a factor of degree $k$, then the falling factorial ${}_k(n+a)$ satisfies a bound on its largest prime divisor, with the falling factorial ${}_k(m)$ ($m,k\in \N$) being defined by
\[
{}_k(m) \coloneqq m(m-1)(m-2)\cdots (m-k+1).
\]
  \item Lower bounds on the largest prime divisor of these falling factorials give an upper bound on $k$ that is independent of $n$.
  \item The remaining divisor degrees may occur only in a density-zero set.
\end{enumerate}
We will establish irreducibility using work by Shorey and Tijdeman:
\begin{lemma}\label{lem:ST_smoothness}
  If $P_{n,a}(x)$ has a factor in $\Q[x]$ of degree $1\le k\le \frac{n}{2}$, then the largest prime divisor of ${}_k(n+a)$ is less than or equal to $k+a$.
\end{lemma}
\begin{proof}
  For $k\ge2$, this is \cite[Corollary~4.3, p.~354]{shorey-tijdeman}. Their proof of Lemma~4.2 \cite[pp.~354--355]{shorey-tijdeman} also applies unchanged to $k=1$; we include a direct argument.

  Suppose $P_{n,a}$ has a rational root $t$, which is integral by monicity. If a prime $p>a+1$ divides $n+a$, then $P_{n,a}(x)\equiv x^n\pmod p$, so $p\mid t$. For $1\le j\le n$,
  \[
  v_p((a+j)!)<\frac{a+j}{p-1}\le j.
  \]
  Thus every nonconstant term of $F_{n,a}(t)$ has positive $p$-adic valuation, whereas its constant term $1/a!$ is a unit, a contradiction.
\end{proof}
\subsection{From Smoothness Bounds to Degree Bounds}\label{subsec:smoothness}
There is a substantial literature on finding lower bounds on the largest prime divisor of ${}_k(x)$. We will denote the largest prime divisor of $N$ by $P^+(N)$.
\begin{lemma}\label{lem:fall-fact-smooth}
  There exists an absolute constant $1<c<2$ such that every pair of positive integers $(x,k)$ with $x\ge 2k$, outside a finite exceptional set, satisfies
  \[
  P^+({}_k(x))>ck.
  \]
\end{lemma}

The case $c = 1$ was proved by Sylvester in \cite[Part~I, \S\S2--3, pp.~4--19]{sylvester1892} with no exceptions. Moser \cite[Lemma~1, p.~167]{moser1963} improved this to $c = \frac{11}{10}$. Hanson \cite[Theorem, p.~195]{hanson1973} then improved this to $c = \frac{3}{2}$, with 3 exceptions $(x,k) = (4,2),(9,2),(10,5)$. For $k>2$, Laishram–Shorey \cite[Corollary~1, p.~300]{laishram-shorey2005} gave another improvement, $c = 1.8$, with a set of 14 exceptions.

This brings us to the main lemma in this section:
\begin{lemma}\label{lem:degree-bound}
  For every fixed integer shift $a>0$, there exists a constant $1<c<2$ such that, for all sufficiently large $n$, if $P_{n,a}(x)$ is reducible over $\Q$, its smallest irreducible factor is of degree $k<\frac{a}{c-1}$.
\end{lemma}
\begin{proof}
  Let $1<c<2$ be the absolute constant from Lemma \ref{lem:fall-fact-smooth}.
  First, note that the smallest irreducible factor degree $k$ satisfies $1\le k\le \frac{n}{2}<\frac{n+a}{2}$. Hence, we can use Lemma \ref{lem:fall-fact-smooth} to see that for sufficiently large $n$,
  \[
  P^+({}_k(n+a))>ck.
  \]
  However, Lemma \ref{lem:ST_smoothness} implies that $P^+({}_k(n+a)) \le k+a$. Combining the two, we get
  \[
  ck<k+a \implies k<\frac{a}{c-1},
  \]
  as desired.
\end{proof}

As an example, we can apply \cite[Theorem, p.~195]{hanson1973} to take $c = \frac{3}{2}$, and hence, if $P_{n,a}(x)$ is reducible, its minimal factor has degree $k<2a$ for all $n>10-a$ (as the exceptional set consists of the pairs $(n+a,k) = (4,2),(9,2),(10,5)$). This argument is essentially the one used by Shorey and Tijdeman in \cite[\S4]{shorey-tijdeman}. In fact, \cite[Theorem~1.1]{shorey-tijdeman} gives more: if $2\le k\le \frac n2$ and $a\le\frac{3k}{2}$, then $P_{n,a}$ has a factor of degree $k$ only for finitely many explicit triples $(n,k,a)$.

We are now ready for the final lemma needed for the proof of Theorem \ref{thm:density-irreducible}.

\subsection{Small Factors Occur with Density 0}\label{subsec:small-factors}
For $k\in\N$, write $[k]=\{1,\ldots,k\}$.
\begin{lemma}\label{lem:small-factors-density-0}
  For a fixed integer $a>0$ and any fixed $M>1$, let
  \[
  S\coloneqq \{n\in \N : \exists f(x)\in\Q[x] \text{ such that } f(x)\mid P_{n,a}, 1\le\deg f(x)<M\}.
  \]
Then $S$ has natural density zero; that is,
\[
\lim_{k\to\infty}\frac{1}{k}\left | S\cap [k]\right | = 0.
\]
\end{lemma}

To prove this, we first note the following:
\begin{proposition}\label{prop:S-contained-in-smooth}
For the set $S$ defined in Lemma \ref{lem:small-factors-density-0}, we have
\[
S\subset \{ n: P^+(n+a) < M+a \}.
\]
\end{proposition}
\begin{proof}
  Let $n\in S$. If $n<M$, the conclusion is immediate from $P^+(n+a)\le n+a<M+a$. Otherwise, $P_{n,a}$ has a proper factor of degree less than $M$. Choose a smallest irreducible factor of $P_{n,a}$; its degree $k$ satisfies $1\le k\le n/2$ and $k<M$. By Lemma \ref{lem:ST_smoothness} and the fact that $(n+a)\mid{}_k(n+a)$, we get
  \[
  P^+(n+a)\le P^+({}_k (n+a))\le k+a<M+a.
  \]
\end{proof}
For example, consider $a=1$ and $M=2$. In this case, $S$ is the set of degrees $n$ for which $P_{n,1}(x)$ has a linear factor. Proposition \ref{prop:S-contained-in-smooth} implies that $S$ is contained in the set of Mersenne numbers, or numbers of the form $n = 2^k-1$ (as $P^+(n+a) < M+a \implies P^+(n+1) < 3 \implies n+1 = 2^k$). This case was also explored by Schur in the second part of his paper \cite[Part~II, pp.~370--391]{schur1929}.

Note that the Mersenne numbers clearly have natural density zero, which implies Lemma \ref{lem:small-factors-density-0} for $a=1$ and $M=2$. We will prove it in general by establishing this density claim for all $a,M>0$.

\begin{proof}[Proof of Lemma~\ref{lem:small-factors-density-0}]
  Fix $C>0$, and let $q_1,\ldots,q_r$ be the primes below $C$. Every $m\le X$ with $P^+(m)<C$ has the form $m=\prod_{i=1}^r q_i^{e_i}$, where $0\le e_i\le \log X/\log q_i$.
  Hence there are $O_C((\log X)^r)=o(X)$ such integers, so $\{m:P^+(m)<C\}$ has natural density zero.
  Taking $C=M+a$, we find that $\{n:P^+(n+a)<M+a\}$ has density zero, and Proposition \ref{prop:S-contained-in-smooth} then implies that $S$ has natural density zero.
\end{proof}
Finally, we are ready to prove Theorem \ref{thm:density-irreducible}.
\begin{proof}[Proof of Theorem~\ref{thm:density-irreducible}]
  Let $a$ be a positive integer, and let $c>1$ be as in Lemma \ref{lem:fall-fact-smooth}. Define the set $U_a$ by
  \[U_a \coloneqq \{n\in\N:P_{n,a}(x) \text{ factors over } \Q \}.\]
  Let us split $U_a$ into two parts:
  \[
  \begin{aligned}
  U_a&=S\cup T,\\
  S& \coloneqq \left\{\begin{aligned}
    &n\in \N : \exists f(x)\in\Q[x] \text{ such that } f(x)\mid P_{n,a},\\
    &1\le\deg f(x)<\min(\deg P_{n,a},\frac{a}{c-1})
  \end{aligned}\right\},\\
  T& \coloneqq U_a\setminus S.
  \end{aligned}
  \]
  Lemma \ref{lem:degree-bound} implies $T\subset [m]$ for some $m$, as $n\in T$ implies $P_{n,a}$ has a minimal nontrivial factor of high degree, which does not occur for sufficiently large $n$. Lemma \ref{lem:small-factors-density-0} implies $S$ has natural density zero.

  Hence, $U_a$ also has density zero, meaning that $P_{n,a}(x)$ is irreducible for a density-one set of $n$, as desired.
\end{proof}

\pagebreak
\section{Primitivity of \texorpdfstring{$G_{n,a}$}{G(n,a)}}\label{sec:primitive}
In this section, we will establish the following theorem:
\begin{theorem}\label{thm:primitivity}
  Let $a>0$ be some fixed integer. Then, for any sufficiently large $n$ for which $P_{n,a}(x)$ is irreducible, the group $G_{n,a} \coloneqq \Gal(F_{n,a}/\Q)$ is primitive.
\end{theorem}
Using Newton polygons, we will establish the existence of two cycles in the group, and use them to rule out any nontrivial block structure.
\begin{definition}
  For a prime $p$ and $f(x)=\sum_i c_ix^i\in\Q[x]$, the \emph{$p$-adic Newton polygon} $\NP_p(f)$ is the lower convex hull of the points $(i,v_p(c_i))$. 
\end{definition}
The slopes of the edges of these polygons encode the valuations of the roots in $\overline{\Q}_p$, and hence constrain the degrees of the factors of $f(x)$. See \cite[Ch.~VI, \S3, Thm.~3.1, p.~99]{cassels-local-fields}.
Let us begin with the first cycle:
\subsection{\texorpdfstring{$G_{n,a}$}{G(n,a)} Contains a \texorpdfstring{$q$}{q}-Cycle}\label{subsec:q-cycle}
\begin{lemma}\label{lem:q-cycle}
  Let $n\ge2a$, and let $q$ be a prime in the range 
  \[
  \frac{n+a}{3}<q<\frac{n}{2}.
  \]
  Then there exists a $q$-cycle in $G_{n,a}$.
\end{lemma}
To show this, let us first establish the Newton polygon of $P_{n,a}(x)$ at $q$, and use it to get a factorization over $\Q_q$. Note that $a\le \frac{n+a}{3}<q$, and that $n+a$ is written as $2q+b$ in base $q$ (with $0\le b<q$). A direct calculation of the coefficient valuations gives
\[
v_q(b_j)=
\begin{cases}
2 & 0\le j<q-a,\\
1 & q-a\le j<2q-a,\\
0 & 2q-a\le j\le n.
\end{cases}
\]
Consequently,
\[
  \NP_q(P_{n,a})
  \quad\text{with vertices}\quad
  (0,2),\ (q-a,1),\ (2q-a,0),\ (n,0).
\]
The edges have the following slopes and horizontal lengths:
\[
s_1 = \frac{-1}{q-a},\ \ell_1 = q-a, \qquad s_2 = \frac{-1}{q},\ \ell_2 = q, \qquad s_3 = 0,\ \ell_3 = b = n+a-2q.
\]
The coefficient points and their lower convex hull are shown schematically in \cref{fig:q-newton-polygon}.
\begin{figure}[h]
  \centering
  \begin{tikzpicture}[x=0.76cm,y=1.12cm,font=\small]
    \draw[-{Latex[length=2mm]}] (-0.45,-0.45) -- (12.75,-0.45)
      node[right] {$j$};
    \draw[-{Latex[length=2mm]}] (-0.35,-0.55) -- (-0.35,2.65)
      node[above] {$v_q(b_j)$};
    \foreach \y in {0,1,2} {
      \draw (-0.42,\y) -- (-0.28,\y);
      \node[left] at (-0.42,\y) {$\y$};
    }
    \foreach \x/\lab in {0/0,3.8/{q-a},8.8/{2q-a},12/n} {
      \draw (\x,-0.52) -- (\x,-0.38);
      \node[below] at (\x,-0.52) {$\lab$};
    }

    \draw[blue!70!black,very thick]
      (0,2) -- (3.8,1) -- (8.8,0) -- (12,0);

    \foreach \x in {0,0.55,1.10,2.70,3.25} \fill (\x,2) circle (1.65pt);
    \node[fill=white,inner sep=0.5pt] at (1.90,2) {$\cdots$};
    \foreach \x in {3.8,4.40,5.00,7.55,8.15} \fill (\x,1) circle (1.65pt);
    \node[fill=white,inner sep=0.5pt] at (6.30,1) {$\cdots$};
    \foreach \x in {8.8,9.35,9.90,10.95,11.50,12} \fill (\x,0) circle (1.65pt);
    \node[fill=white,inner sep=0.5pt] at (10.45,0) {$\cdots$};

    \foreach \x/\y in {0/2,3.8/1,8.8/0,12/0}
      \filldraw[fill=blue!70!black,draw=white,line width=0.35pt]
        (\x,\y) circle (2.35pt);
    \node[below] at (1.75,1.54) {$-\frac{1}{q-a}$};
    \node[below] at (6.35,0.49) {$-\frac{1}{q}$};
    \node[above] at (10.65,0.08) {$0$};
    \node[anchor=west] at (8.95,2.25)
      {$-\frac1{q-a}<-\frac1q<0$};

    \draw[<->] (0,-1.02) -- (3.8,-1.02)
      node[midway,fill=white,inner sep=1pt] {$q-a$};
    \draw[<->] (3.8,-1.02) -- (8.8,-1.02)
      node[midway,fill=white,inner sep=1pt] {$q$};
    \draw[<->] (8.8,-1.02) -- (12,-1.02)
      node[midway,fill=white,inner sep=1pt] {$b$};
  \end{tikzpicture}
  \caption{A sketch of the $q$-adic Newton polygon of $P_{n,a}$. The dots are the points $(j,v_q(b_j))$. The polygon has three edges, of slopes $-\frac1{q-a}$, $-\frac1q$ and $0$, and horizontal lengths $q-a$, $q$ and $b$.}
  \label{fig:q-newton-polygon}
\end{figure}
Notice that $b < q$, precisely because $n+a <3q$.
This implies the following:
\begin{lemma}
  Let $n\ge2a$, and let $q$ be a prime in the range 
  \[
  \frac{n+a}{3}<q<\frac{n}{2}.
  \]
  Then $P_{n,a}(x)$ factors over $\Q_q$ as follows:
  \[
  P_{n,a}(x) = h_q(x)h_{q-a}(x)h_b(x).
  \]
  Here $h_q(x)$ is Eisenstein of degree $q$, and its roots in $\overline{\Q_q}$ have $q$-adic valuation $\frac{1}{q}$; likewise, $h_{q-a}(x)$ is Eisenstein of degree $q-a$, and its roots in $\overline{\Q_q}$ have $q$-adic valuation $\frac{1}{q-a}$. The roots of $h_b(x)$ in $\overline{\Q_q}$ have $q$-adic valuation $0$.
\end{lemma}
\begin{proof}
  This is immediate from Newton factorization.
\end{proof}
Using this, we can prove our lemma.
\begin{proof}[Proof of Lemma~\ref{lem:q-cycle}]
  Let us denote the sets of roots of $h_q,h_{q-a},h_b$ over $\Q_q$ by $R_q,R_{q-a},R_b$.
  
  Since $h_q(x)$ is Eisenstein, it is irreducible, and hence $G_{n,a}^{(q)}  \coloneqq \Gal(P_{n,a}/\Q_q)$ acts transitively on its roots $R_q$. Indeed, its quotient $\Gal(h_q/\Q_q)$ does so. Then $q =\deg(h_q)$ divides the order of this quotient, and hence $q\mid \left|G_{n,a}^{(q)}\right|$.
  By Cauchy's theorem, $G_{n,a}^{(q)}$ has an element $\sigma$ of order $q$. This implies that the orbit of $\langle\sigma\rangle$ on any root has size $1$ or $q$. As the valuations of $R_q,R_{q-a},R_b$ are different, the restriction of $\sigma$ to each of those sets is a permutation of that set. However, as $q-a,b<q$ and $q$ is prime, $\sigma$ fixes all roots in $R_{q-a}$ and $R_b$. It cannot fix all the roots of $R_q$, lest it be trivial. Thus $\sigma$ is a $q$-cycle. Finally, since $\Gal(P_{n,a}(x)/\Q_q)\le G_{n,a}$, we have $\sigma\in G_{n,a}$, which proves the claim.
\end{proof}

\subsection{\texorpdfstring{$G_{n,a}$}{G(n,a)} Contains a \texorpdfstring{$(p-a)$}{(p-a)}-Cycle}\label{subsec:p-a-cycle}
\begin{lemma}\label{lem:p-a-cycle}
  Let $n\ge a$, and let $p$ be a prime in the range
  \[
  \frac{n}{2}+a<p<n+a.
  \]
  Then there exists a $(p-a)$-cycle in $G_{n,a}$.
\end{lemma}
Again, we must begin by calculating $\NP_p(P_{n,a})$ and using it to factorize over $\Q_p$. Note that $a\le \frac{n+a}{2}<p$, and that $n+a$ is written as $p+r$ in base $p$ (with $0\le r<p$). A direct calculation gives
\[
v_p(b_j)=
\begin{cases}
1 & 0\le j<p-a,\\
0 & p-a\le j\le n.
\end{cases}
\]
Thus \(\NP_p(P_{n,a})\) has vertices
\[
  (0,1),\ (p-a,0),\ (n,0).
\]
The edges have the following slopes and horizontal lengths:
\[
s_1 = \frac{-1}{p-a},\ \ell_1 = p-a, \qquad s_2 = 0,\ \ell_2 = r = n+a-p.
\]
The coefficient points and their lower convex hull are shown schematically in \cref{fig:p-newton-polygon}.
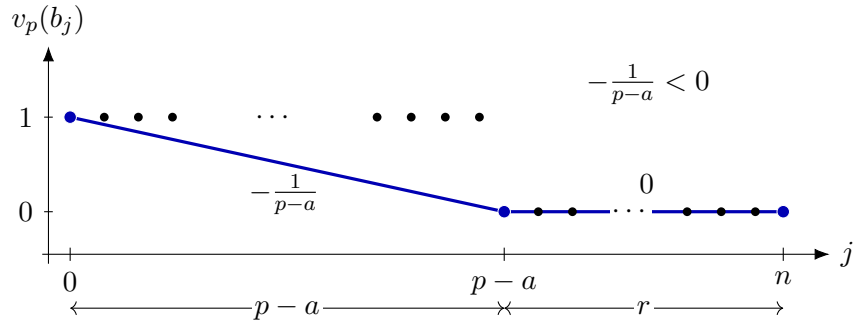
\begin{figure}[ht]
  \centering
  \begin{tikzpicture}[x=0.82cm,y=1.25cm,font=\small]
    \draw[-{Latex[length=2mm]}] (-0.45,-0.45) -- (12.25,-0.45)
      node[right] {$j$};
    \draw[-{Latex[length=2mm]}] (-0.35,-0.55) -- (-0.35,1.75)
      node[above] {$v_p(b_j)$};
    \foreach \y in {0,1} {
      \draw (-0.42,\y) -- (-0.28,\y);
      \node[left] at (-0.42,\y) {$\y$};
    }
    \foreach \x/\lab in {0/0,7/{p-a},11.5/n} {
      \draw (\x,-0.52) -- (\x,-0.38);
      \node[below] at (\x,-0.52) {$\lab$};
    }

    \draw[blue!70!black,very thick] (0,1) -- (7,0) -- (11.5,0);

    \foreach \x in {0,0.55,1.10,1.65,4.95,5.50,6.05,6.60}
      \fill (\x,1) circle (1.65pt);
    \node[fill=white,inner sep=0.5pt] at (3.30,1) {$\cdots$};
    \foreach \x in {7,7.55,8.10,9.95,10.50,11.05,11.5}
      \fill (\x,0) circle (1.65pt);
    \node[fill=white,inner sep=0.5pt] at (9.05,0) {$\cdots$};

    \foreach \x/\y in {0/1,7/0,11.5/0}
      \filldraw[fill=blue!70!black,draw=white,line width=0.35pt]
        (\x,\y) circle (2.35pt);
    \node[below] at (3.45,0.51) {$-\frac{1}{p-a}$};
    \node[above] at (9.30,0.08) {$0$};
    \node[anchor=west] at (8.15,1.35) {$-\frac1{p-a}<0$};

    \draw[<->] (0,-1.02) -- (7,-1.02)
      node[midway,fill=white,inner sep=1pt] {$p-a$};
    \draw[<->] (7,-1.02) -- (11.5,-1.02)
      node[midway,fill=white,inner sep=1pt] {$r$};
  \end{tikzpicture}
  \caption{A sketch of the $p$-adic Newton polygon of $P_{n,a}$. The dots are the points $(j,v_p(b_j))$. The polygon has two edges, of slopes $-\frac1{p-a}$ and $0$, and horizontal lengths $p-a$ and $r$.}
  \label{fig:p-newton-polygon}
\end{figure}
This implies the following:
\begin{lemma}\label{lem:p-a-factor}
  Let $n\ge a$, and let $p$ be a prime in the range
  \[
  \frac{n}{2}+a<p<n+a.
  \]
  Then $P_{n,a}(x)$ factors over $\Q_p$ as follows:
  \[
  P_{n,a}(x) = h_{p-a}(x)h_r(x).
  \]
  Here $h_{p-a}(x)$ is Eisenstein of degree $p-a$, and its roots in $\overline{\Q_p}$ have $p$-adic valuation $\frac{1}{p-a}$. The roots of $h_r(x)$ in $\overline{\Q_p}$ have $p$-adic valuation $0$.
\end{lemma}
\begin{proof}
  This is immediate from Newton factorization.
\end{proof}
Unlike the previous section, this is not enough to get a $(p-a)$-cycle, as $p-a$ is not necessarily prime. To establish this, we must use more than just the $p$-adic valuations of the roots. We start with the following proposition:
\begin{proposition}\label{prop:Pna-mod-p}
  Let $n\ge a$, and let $p$ be a prime in the range
  \[
  \frac{n}{2}+a<p<n+a.
  \]
  Then, with $r = n+a-p$, 
  \[
  P_{n,a}(x) \equiv r!x^{p-a}\sum_{j = 0}^r \frac{x^j}{j!} \pmod p.
  \]
\end{proposition}
\begin{proof}
  Write $b_i=(n+a)!/(a+i)!$. Since $n+a=p+r<2p$, we have $b_i\equiv0\pmod p$ when $i<p-a$. If $i=p-a+j$ with $0\le j\le r$, then
  \[
  b_i=\frac{(p+r)!}{(p+j)!}=\prod_{t=j+1}^r(p+t)\equiv\frac{r!}{j!}\pmod p.
  \]
  Summing these coefficients gives the stated congruence.
\end{proof}
Note that $h_r(x)\equiv r!\sum_{j = 0}^r \frac{x^j}{j!} \pmod p$. This implies the following fact:
\begin{lemma}\label{lem:Pna-unram-tame-split}
  The splitting fields of $h_r$ and $h_{p-a}$ over $\Q_p$ are unramified and tamely ramified, respectively.
\end{lemma}
\begin{proof}
  We have $\overline{h_r}(x)=r!\sum_{j=0}^r x^j/j!$ and
  \[
  \overline{h_r}(x)-\overline{h_r}'(x)=x^r.
  \]
  Since $\overline{h_r}(0)=r!\not\equiv0\pmod p$, every root $\alpha$ of $\overline{h_r}$ is nonzero, and therefore $\overline{h_r}'(\alpha)=-\alpha^r\ne0$. Thus $\overline{h_r}$ is squarefree, so Hensel's lemma shows that $h_r/\Q_p$ is unramified. The polynomial $h_{p-a}$ is Eisenstein of degree $1\le p-a<p$, so $h_{p-a}/\Q_p$ is tamely ramified.
\end{proof}
We are ready to prove the lemma.
\begin{proof}[Proof of Lemma~\ref{lem:p-a-cycle}]
  
  Let $K/\Q_p$ be an unramified extension in which $h_r$ factors completely. $\Gal(P_{n,a}/K)$ is a subgroup of $\Gal(P_{n,a}/\Q_p)$ (extension of the ground field always gives a subgroup). But as $h_r$ has all roots in $K$, $\Gal(P_{n,a}/K)$ is $\Gal(h_{p-a}/K)$. The group $\Gal(h_{p-a}/K)$ in turn contains the subgroup $\Gal(h_{p-a}/\Q_p^{\mathrm{unr}})$. Here $h_{p-a}/\Q_p^{\mathrm{unr}}$ remains an Eisenstein polynomial of degree $d=p-a$ prime to $p$, so its splitting field is $\Q_p^{\mathrm{unr}}(p^{1/d})$, which is cyclic of degree $d$. Hence from transitivity,
$\Gal(h_{p-a}/\Q_p^{\mathrm{unr}})$ contains (and indeed is generated by) a $d$-cycle.
\end{proof}

\subsection{Eliminating Block Structure}\label{subsec:block-structure}
To establish Theorem \ref{thm:primitivity}, we must eliminate the possibility of the group $G_{n,a}$ being imprimitive. Given that we assume it is transitive (from our assumption that $P_{n,a}$ is irreducible), we only need to rule out the possibility of a ``block structure''. A block structure is a partition of the $n$ roots of $P_{n,a}$ into $1<d<n$ disjoint sets $B_1,\dots,B_d$ of equal size $\frac{n}{d}$, such that any $\sigma \in G_{n,a}$ satisfies $\sigma (B_i) = B_j$. We will use the cycles we established in the previous sections to rule out any such block structure.
\begin{lemma}\label{lem:d-le-2}
  Let $G\le S_n$ be a transitive group. Assume further that it contains a $q$-cycle, where $q$ is prime and $\frac{n}{3}<q<\frac{n}{2}$. Then $G$ is either primitive or has $d = 2$ blocks.
\end{lemma}
\begin{proof}
  Without loss of generality, let the $q$-cycle in $G$ be $\sigma = (1,2,\dots,q)$. Assume further that $G$ has a block structure with $d>2$ blocks and $1\in B_1$. In the first case, $\sigma(B_1) = B_1$, so $2 = \sigma(1)\in B_1$. More generally, for any $1<i\le q$, if $i-1\in B_1$, then $i = \sigma(i-1)\in B_1$. This implies $[q]\subseteq B_1$, so $q\le |B_1|=n/d\le n/3$, contradicting $q>n/3$.
  The other possible case is $\sigma(B_1) \neq B_1$. Then the orbit of $B_1$ under $\langle\sigma\rangle$ has size $q$, rather than $1$. Note that for all $i\in [q]$, $i\in B_i$. Another implication is $d\ge q$, which, combined with $d\mid n$ and $q>\frac{n}{3}$, implies $d = \frac{n}{2}$. This implies $|B_1| = 2$. Say $B_1 = \{1,s\}$. Since $s>q$ ($B_1 \cap [q] = \{1\}$), it is fixed by $\sigma$. But this contradicts $\sigma(B_1) = B_2$.
\end{proof}

The $(p-a)$-cycle will rule out the $d = 2$ case. 
\begin{lemma}\label{lem:d-neq-2}
  Let $G\le S_n$ be a transitive group. Assume further that it contains a $(p-a)$-cycle, where $p$ is prime and $\frac{n}{2}+a<p<n+a$. Then $G$ cannot have a block structure with $d = 2$ blocks.
\end{lemma}
\begin{proof}
  Without loss of generality, let the $(p-a)$-cycle in $G$ be
  \[
    \sigma = (1,2,\dots,p-a).
  \]
  Assume further that $G$ has a block structure with $d=2$ blocks $B_1,B_2$, with $1\in B_1$. If $\sigma(B_1) = B_1$, then $i = \sigma^{i-1}(1)\in \sigma (B_1) = B_1$ for all $i\in [p-a]$. But $p-a>|B_1| = \frac{n}{2}$, a contradiction. If $\sigma(B_1)= B_2$, then $\sigma(B_2) = B_1$, which implies $\sigma$ has no fixed points. But since $p-a<n$, it does. Again, this is a contradiction.
\end{proof}
Hence we are ready to prove Theorem \ref{thm:primitivity}. 
For fixed $a$, the prime number theorem supplies primes in both intervals of Lemmas~\ref{lem:q-cycle} and~\ref{lem:p-a-cycle} for every sufficiently large $n$.
\begin{proof}[Proof of Theorem~\ref{thm:primitivity}]
  Let $a>0$, and let $n>2a$ be large enough that both intervals in Lemmas~\ref{lem:q-cycle} and~\ref{lem:p-a-cycle} contain a prime, and such that $P_{n,a}$ is irreducible. Lemma \ref{lem:q-cycle} implies the existence of a $q$-cycle, and Lemma \ref{lem:p-a-cycle} implies the existence of a $(p-a)$-cycle. As $G_{n,a}$ is transitive, these make both Lemmas \ref{lem:d-le-2} and \ref{lem:d-neq-2} apply, which together imply $G_{n,a}$ is primitive, as desired.
\end{proof}
Note the immediate corollary of Theorems \ref{thm:primitivity} and \ref{thm:density-irreducible}:
\begin{corollary}\label{cor:density-1-primitive}
  For a fixed integer $a>0$, the group $G_{n,a} = \Gal(P_{n,a}/\Q)$ is primitive for a density-one set of $n$.
\end{corollary}
This is enough to show the following proposition:
\begin{proposition}\label{prop:A_n-inclusion}
  For a fixed integer $a>0$, the group $G_{n,a} = \Gal(P_{n,a}/\Q)$ contains $A_n$ as a subgroup for a density-one set of $n$. In the language of \cite[Question~1.1 and the following paragraph, pp.~1--2]{katz-rivin}, $\mathcal{N}(f_a) = 0$ for $f_a(x) \coloneqq \sum_{i=0}^{\infty}\frac{x^i}{(i+a)!}$.
\end{proposition}

\begin{proof}
  Using Corollary \ref{cor:density-1-primitive} and Lemma \ref{lem:q-cycle}, we have that $G_{n,a}$ is primitive and contains a $q$-cycle for a density-one set of $n$. The number $q$ is prime and, for large enough $n$, satisfies $q<n-2$. Hence, by Jordan's theorem \cite[Theorem~3.3E, p.~84]{dixon-mortimer}, one gets that $A_n\le G_{n,a}$.
\end{proof}

\section{From \texorpdfstring{$A_n$}{A(n)} to \texorpdfstring{$S_n$}{S(n)}}\label{sec:alternating-symmetric}
In this section, we wish to finally show Theorem \ref{thm:main-target}.
To do so, we need only establish that $\Disc(P_{n,a})\notin \Q^{\times 2}$ for a density-one set of $n$. We will begin by finding the sign of the discriminant, and its $p$-adic valuation for some range of primes $p$. These will settle the case of odd $a$, and some congruence classes of $n$. Next, we will find the value of $\frac{\Disc(P_{n,a})}{p^{p-a-1}}$ mod $p$ for the same range of primes. The Legendre symbols for these values simplify considerably, and for the discriminant to be square, they must all equal $1$, for all primes $p$ in the relevant range. This is then trivially ruled out for the case $n\equiv 1 \pmod 4$. The remaining case $n\equiv 0 \pmod 4$ is dealt with by showing that a positive upper density of square discriminants implies many highly biased Dirichlet characters. Such characters are then ruled out using a zero-density argument for zeros of Dirichlet $L$-functions. A simpler proof conditional on GRH is also mentioned.

\subsection{The Sign of the Discriminant}\label{subseq:disc-sign}
\begin{lemma}\label{lem:disc_sign}
  For any fixed $a>0$, and for a density-one set of integers $n$:
  \[
  \sign(\Disc(P_{n,a})) = \sign(\Disc(F_{n,a})) = \begin{cases}
    1 & n\equiv 0,1 \pmod 4\\
    -1 & n\equiv 2,3 \pmod 4
  \end{cases}.
  \]
\end{lemma}
This will be established using a well-known result of Brill \cite[p.~87]{brill1877}, connecting the sign of the discriminant to a count of nonreal conjugate root pairs.
Let $f(x)\in\R[x]$ satisfy $\Disc(f)\neq0$, and assume it has $s$ pairs of nonreal conjugate roots. Then:
  \[
  \sign(\Disc(f)) = (-1)^s.
  \]
To count the conjugate pairs, we will count the real zeros of $F_{n,a}$. This will be done using the following construction:
\[
E_{n,a}(x)\coloneqq x^aF_{n,a}(x) = \sum_{j = a}^{n+a}\frac{x^j}{j!}.
\]
Note that 
\[E_{n,a}' = E_{n,a-1}, E_{n,0} = F_{n,0}, E_{n,0}' = E_{n-1,0}. \]
Also note, as $F_{n,a}(0)\neq0$, the number of real roots of $F_{n,a}$ is equal to the number of nonzero real roots of $E_{n,a}$. Finally, it is clear from the positivity of the coefficients of $E_{n,a}$ that all its real nonzero roots are negative. These are counted in the following proposition:
\begin{proposition}\label{prop:Ena_roots}
  \[
  \#\text{ nonzero real roots of } E_{n,a}\text{ counted with multiplicity } = \begin{cases}
  0 & n \text{ even}\\
  1 & n \text{ odd}
  \end{cases}.
  \]
\end{proposition}
This is proven by induction on $a$, and then by induction on $n$.
\begin{proof}
Assume the proposition holds for some $a\ge 0$ and for all $n \ge 0$. If $n$ is even, Rolle's theorem shows that a negative root $\rho<0$ of $E_{n,a+1}$ forces a negative root $\rho'\in(\rho,0)$ of $E_{n,a+1}' = E_{n,a}$. This is a contradiction. If $n$ is odd, two negative roots $\rho_1\le\rho_2<0$ of $E_{n,a+1}$, possibly counted with multiplicity, would imply two negative roots $\rho_1'\in[\rho_1,\rho_2],\rho_2'\in(\rho_2,0)$ of $E_{n,a+1}' = E_{n,a}$. This is again a contradiction. Thus $E_{n,a+1}$ has at most one negative root. As $F_{n,a+1}(x)$ is of odd degree, it must have some real root. As $F_{n,a+1}(0)>0$, this root is nonzero. This is also a root of $E_{n,a+1} = x^{a+1}F_{n,a+1}$, completing the induction step on $a$.

So it remains to show the proposition holds for $a=0$ and all $n \ge 0$. 
Assume it holds for $E_{n,0}$ for some even $n\ge 0$. As $E_{n,0}(0) = 1>0$, and  $E_{n,0}$ has no real zeros, this implies $E_{n,0}(x)>0$ for all $x$. As $E_{n+1,0}' = E_{n,0}$, we have that $E_{n+1,0}$ is strictly increasing, and as it has odd degree, it has a single real (clearly nonzero) root $\rho$. Moreover, as $E_{n+2,0}'=E_{n+1,0}$, $E_{n+2,0}$ has a single global minimum at $x = \rho$. Finally, as $E_{n+2,0}(x)=E_{n+1,0}(x)+\frac{x^{n+2}}{(n+2)!}$, we have that $E_{n+2,0}(\rho)=\frac{\rho^{n+2}}{(n+2)!}>0$, implying $E_{n+2,0}(x)>0$ for all $x$. This completes the induction step on $n$.

The base case $E_{0,0}(x) = 1$ is trivial.
\end{proof}
We are ready to prove Lemma \ref{lem:disc_sign}:
\begin{proof}[Proof of Lemma~\ref{lem:disc_sign}]
  In the density-one set of integers $n$ for which $F_{n,a},P_{n,a}$ are irreducible (see Theorem \ref{thm:density-irreducible}), the polynomials are squarefree, and hence
  \[\Disc(P_{n,a}) = ((n+a)!)^{2n-2} \Disc(F_{n,a})\neq 0.\]
  This allows us to use the root-count formula for the sign of the discriminant. $F_{n,a}(x)$ has $n$ roots, and hence the number of conjugate pairs is $s = \frac{n-t}{2}$ where $t$ is the number of real roots. From $F_{n,a}(0)\neq0$, the number of real roots of $F_{n,a}$ equals the number of nonzero real roots of $E_{n,a}$. Using Proposition \ref{prop:Ena_roots}, we find that $t = \begin{cases}
  0 & n \text{ even}\\
  1 & n \text{ odd}
  \end{cases}$, implying that $s = \left\lfloor\frac{n}{2}\right\rfloor$, which finally gives the desired:
  \[
    \sign(\Disc(P_{n,a})) = \sign(\Disc(F_{n,a})) = (-1)^s = \begin{cases}
    1 & n\equiv 0,1 \pmod 4\\
    -1 & n\equiv 2,3 \pmod 4
  \end{cases}.
  \]
\end{proof}
This covers the case of $n\equiv 2,3 \pmod 4$, as negative numbers cannot be perfect squares.

\subsection{The Discriminant in \texorpdfstring{$\Q_p$}{Qp}}\label{subseq:disc-in-Qp}
In this section, we will calculate the valuation and angular component (see Definition 
\ref{def:ac}) of the discriminant in $\Q_p$, for some primes $p$. These will be used to show the discriminant is nonsquare with density one.
\begin{definition}\label{def:ac}
  The leading digit, or \emph{angular component} $\ac : \Q_p^\times \to \F_p^\times$ of $x\in\Q_p^\times$ is defined as:
  \[
  \ac(x) \coloneqq \frac{x}{p^{\vp(x)}}\pmod p.
  \]
  Note that $\ac(xy) = \ac(x)\ac(y)$.
\end{definition}
\begin{remark}\label{rem:square-ac}
  For $x\in\Q^\times$, $v_p(x^2)\equiv 0\pmod 2$, and $\left(\frac{\ac(x^2)}{p}\right) = 1$, for any odd prime $p$.
\end{remark}
To compute these local invariants, we will first prove a proposition about Eisenstein polynomials and then reuse the factorization from Subsection \ref{subsec:p-a-cycle}.
\begin{proposition}\label{prop:disc-val-eisen}
  Every monic Eisenstein polynomial $g(x)=\sum_{i=0}^{m}b_ix^i\in\Q_p[x]$ of degree $m$ coprime to $p$ satisfies
  \[
  v_p(\Disc g) = m-1. 
  \]
  Moreover,
  \[
  \ac(\Disc g) \equiv (-1)^{\binom{m}{2}}m^m \left(b_0p^{-1}\right)^{m-1}\pmod p.
  \]
\end{proposition}
\begin{proof}
  Recall that if $\alpha_1,\ldots,\alpha_m$ are the roots of $g(x)$ in $\overline{\Q_p}$, then
  \[
  \Disc(g) = (-1)^{\binom{m}{2}}\prod_{i}g'(\alpha_i).
  \]
  This implies $v_p(\Disc(g)) = \sum_{i}v_p(g'(\alpha_i))$. Note that $v_p(\alpha_i) = \frac{1}{m}$ and that
  \[
  g'(\alpha_i) = m\alpha_i^{m-1}+(m-1)b_{m-1}\alpha_i^{m-2}+\cdots+2b_{2}\alpha_i+b_1.
  \]
  Note that $v_p(m\alpha_i^{m-1}) = \frac{m-1}{m}<1$, and $v_p(kb_{k}\alpha_i^{k-1})\ge 1+\frac{k-1}{m}\ge 1$ for all $k<m$. Hence,
  \[
  v_p(g'(\alpha_i)) = v_p(m\alpha_i^{m-1}) = \frac{m-1}{m}.
  \]
  This finally implies the desired equality:
  \[
  v_p(\Disc(g)) = \sum_{i}v_p(g'(\alpha_i)) = m-1.
  \]
  To calculate $\ac(\Disc(g))$, we again use the fact that $g'(\alpha_i) = m\alpha_i^{m-1} (1+O(p^{\frac{1}{m}}))$. Taking the product, one gets:
  \[
  \prod_{i}g'(\alpha_i) = m^m\prod_{i}\alpha_i^{m-1}(1+O(p^{\frac{1}{m}})) =m^m b_0^{m-1}(1+O(p^{\frac{1}{m}})).
  \]
  As this product is in $\Q_p$, we have that $O(p^{\frac{1}{m}})=O(p)$, which implies:
  \[
  \ac(\prod_{i}g'(\alpha_i))\equiv \frac{\prod_{i}g'(\alpha_i)}{p^{m-1}}   \equiv m^m \left(b_0p^{-1}\right)^{m-1}\pmod p.
  \]
  This finally gives the desired formula:
  \[
  \ac(\Disc g) \equiv (-1)^{\binom{m}{2}}m^m \left(b_0p^{-1}\right)^{m-1}\pmod p.
  \]
\end{proof}

\begin{lemma}\label{lem:disc-valuation-Pna}
  Let $n\ge a$, and let $p$ be a prime in the range
  \[
  \frac{n}{2}+a<p<n+a.
  \]
  Then
  \[
  v_p(\Disc(P_{n,a})) = p-a-1.
  \]
  Moreover,
  \[
  \ac(\Disc(P_{n,a}))\equiv (-1)^{\binom{p-a}{2}+\binom{r}{2}+1}a((a-1)!)^{a}(r!)^r \Res(h_r,h_{p-a})^2.
  \]
\end{lemma}
\begin{proof}
  For monic polynomials $f$ and $g$, the root-product definitions of the discriminant and resultant give
  \[
  \Disc(fg) = \Disc(f)\Disc(g)\Res(f,g)^2.
  \]
  This agrees with \cite[Ch.~12, eqs.~(1.23) and~(1.32), pp.~403, 405]{gelfand-kapranov-zelevinsky1994}, whose discriminant is normalized as $\Delta(f)=(-1)^{\binom{\deg f}{2}}\Disc(f)$.
  Using Lemma \ref{lem:p-a-factor}, we can factor $P_{n,a} = h_{r}h_{p-a}$ in $\Q_p[x]$, implying:
 \[
  \Disc(P_{n,a}) = \Disc(h_{p-a})\Disc(h_r)\Res(h_r,h_{p-a})^2.
  \]
  Coleman shows in \cite[\S~IV, p.~187]{coleman1987} that $\Disc(r!\sum_{j=0}^r \frac{x^j}{j!}) = (-1)^{\binom{r}{2}}(r!)^r$. As this is the value of $h_r$ mod $p$, we have that $\Disc(h_r)\equiv (-1)^{\binom{r}{2}}(r!)^r \pmod p$. Hence,
  \[v_p(\Disc(h_r)) = 0,\qquad \ac(\Disc(h_r))\equiv (-1)^{\binom{r}{2}}(r!)^r.\]
  As to $\Disc(h_{p-a})$, using Proposition \ref{prop:disc-val-eisen}, we have that 
\[
\begin{gathered}
  v_p(\Disc(h_{p-a})) = p-a-1,\\
  \ac(\Disc(h_{p-a}))\equiv (-1)^{\binom{p-a}{2}}(p-a)^{p-a}\left(b_0p^{-1}\right)^{p-a-1},
\end{gathered}
\]
  with $b_0\coloneqq h_{p-a}(0)$. From $P_{n,a}(0) = h_{r}(0)h_{p-a}(0)$, we get:
\[h_{p-a}(0) \equiv \frac{(n+a)!}{a!r!} =\frac{(p+r)!}{a!r!} \equiv \frac{-p}{a!} \pmod{p^2}.\]
This gives:
\[
\frac{\Disc(h_{p-a})}{p^{p-a-1}} \equiv (-1)^{\binom{p-a}{2}}(p-a)^{(p-a)}\left(-\frac{1}{a!}\right)^{p-a-1}.
\]
This can be simplified using Fermat's little theorem and the fact that $a$ is coprime to $p$:
\[
\begin{aligned}
&(p-a)^{p-a}
 \left(-\frac1{a!}\right)^{p-a-1}\\
&\qquad\equiv
 (-a)^{p-a}
 \left(-\frac1{a!}\right)^{p-a-1}\\
&\qquad=
 -a^{p-a}\left(\frac1{a!}\right)^{p-a-1}\\
&\qquad\equiv
 -a^{1-a}(a!)^a
 =-a((a-1)!)^a
 \pmod p.
\end{aligned}
\]
This gives the final result:
\[
\frac{\Disc(h_{p-a})}{p^{p-a-1}} \equiv (-1)^{\binom{p-a}{2}+1}a((a-1)!)^{a}.
\]
  Finally, recalling the definition of $\Res(h_r,h_{p-a})$ as being the product of the differences of the roots of $h_r$ and $h_{p-a}$, the difference in valuations implies 
  \[
  v_p(\Res(h_r,h_{p-a})) = 0,
  \qquad \ac(\Res(h_r,h_{p-a}))\equiv \Res(h_r,h_{p-a}).
  \]
  This gives the final results:
  \[
  v_p(\Disc(P_{n,a})) = p-a-1,
  \]
  \[
  \ac(\Disc(P_{n,a}))\equiv (-1)^{\binom{p-a}{2}+\binom{r}{2}+1}a((a-1)!)^{a}(r!)^r \Res(h_r,h_{p-a})^2.
  \]
\end{proof}

This is enough to establish the main result for odd $a$:

\begin{proof}[Proof of Theorem~\ref{thm:main-target} for odd $a$]
Let $a$ be a fixed odd positive integer. Recall that Proposition \ref{prop:A_n-inclusion} already shows that $A_n\le G_{n,a}$ for a density-one set of $n$ large enough that there exists a prime $p>2$ in the range
\[
  \frac{n}{2}+a<p<n+a.
\]
Then, as shown in Lemma \ref{lem:disc-valuation-Pna}, we have $v_p(\Disc(P_{n,a})) = p-a-1$. Since $p$ and $a$ are both odd, $v_p(\Disc(P_{n,a}))$ is odd, which means $\Disc(P_{n,a})\notin\Q^{\times 2}$ from Remark \ref{rem:square-ac}. This implies $G_{n,a}\neq A_n$, and hence
\[
G_{n,a} = S_n.
\]
\end{proof}
\subsection{The Value of \texorpdfstring{$\left(\frac{\ac(\Disc(P_{n,a}))}{p}\right)$}{the Legendre Symbol of the Angular Component}} \label{subseq:disc-ac-Leg}
We wish to use the second point in Remark \ref{rem:square-ac} to rule out a square discriminant in the case of even $a$. Note the further simplification of taking the Legendre symbol of $\ac(\Disc(P_{n,a}))$, and limiting our scope to the cases $n\equiv 0,1 \pmod 4$ which are not covered by the sign of the discriminant (see Lemma \ref{lem:disc_sign}).
\begin{lemma}\label{lem:ac-disc}
  For an even integer $a>0$, let $n\ge a$, and let $p$ be a prime in the range
  \[
  \frac{n}{2}+a<p<n+a.
  \]
  Then, denoting $r = n+a-p$, we have:
  \[
\left( \frac{\ac(\Disc(P_{n,a}))}{p}\right)
= \begin{cases}
  \left( \frac{-a}{p}\right) & n\equiv 1 \pmod 4\\
  \left( \frac{ar!}{p}\right) & n\equiv 0 \pmod 4\\
\end{cases}.
\]
\end{lemma}
\begin{proof}
  From Lemma \ref{lem:disc-valuation-Pna}, we have
  \[
  \ac(\Disc(P_{n,a}))\equiv (-1)^{\binom{p-a}{2}+\binom{r}{2}+1}a((a-1)!)^{a}(r!)^r \Res(h_r,h_{p-a})^2.
  \]
  Note that $\Res(h_r,h_{p-a})^2$ is a square, and hence does not affect the Legendre symbol. We will now simplify the remaining terms. First, note that $a$ is even, so $((a-1)!)^{a}$ is a square, and hence does not affect the Legendre symbol. Next, we have
  \[
  (-1)^{\binom{p-a}{2}+\binom{r}{2}+1} = (-1)^{\binom{p-a}{2}+\binom{n-(p-a)}{2}+1} = (-1)^{\floor{\frac{p-a}{2}}+\floor{\frac{n-(p-a)}{2}}+1}.
  \]
  Note $p-a$ is odd, and hence $\floor{\frac{p-a}{2}} = \frac{p-a-1}{2}$. If $n\equiv 0 \pmod 4$, one has:
  \[
\begin{aligned}
\floor{\frac{p-a}{2}}
+\floor{\frac{n-(p-a)}{2}}
&=\frac{p-a-1}{2}
  +\frac n2
  +\floor{\frac{-(p-a)}2}\\
&=\frac{p-a-1}{2}
  +\frac n2
  +\frac{a-p-1}{2}\\
&=\frac n2-1
 \equiv1\pmod2.
\end{aligned}
  \]
  If $n\equiv 1 \pmod 4$, we instead have:
  \[
\begin{aligned}
\floor{\frac{p-a}{2}}
+\floor{\frac{n-(p-a)}{2}}
&=\frac{p-a-1}{2}
  +\frac{n-1}{2}
  +\floor{\frac{1-(p-a)}2}\\
&=\frac{p-a-1}{2}
  +\frac{n-1}{2}
  +\frac{a-p+1}{2}\\
&=\frac{n-1}{2}
 \equiv0\pmod2.
\end{aligned}
  \]
  This finally implies:
  \[
(-1)^{\binom{p-a}{2}+\binom{r}{2}+1} = \begin{cases}
  -1 & n\equiv 1 \pmod 4\\
  1 & n\equiv 0 \pmod 4\\
\end{cases}.
\]
Combining these results:
\[
\begin{aligned}
\Leg{\ac(\Disc(P_{n,a}))}
&= \Leg{(-1)^{\binom{p-a}{2}+\binom{r}{2}+1}}
   \Leg{a((a-1)!)^{a}}\\
&\qquad {}\cdot\Leg{(r!)^r}
   \Leg{\Res(h_r,h_{p-a})^2}\\
&= \Leg{(-1)^{\binom{p-a}{2}+\binom{r}{2}+1}}
   \Leg{a}\Leg{(r!)^r}\\
&= \begin{cases}
  \left( \frac{-a}{p}\right) & n\equiv 1 \pmod 4\\
  \left( \frac{ar!}{p}\right) & n\equiv 0 \pmod 4\\
\end{cases}.
\end{aligned}
\]
\end{proof}
Note that for $n\equiv 1 \pmod 4$ we have that $\Leg{\ac(\Disc(P_{n,a}))}$ is independent of $n$, and its top argument is independent of $p$.
This allows us to resolve that congruence class of values of $n$:
\begin{corollary}\label{cor:even-a-1mod4}
   Fix an even integer $a>0$. Then $G_{4k+1,a} = S_{4k+1}$ for a density-one set of $k$.
\end{corollary}
\begin{proof}
  In Proposition \ref{prop:A_n-inclusion} we have seen that $A_{4k+1}\le G_{4k+1,a}$ for a density-one set of $k$. In Lemma \ref{lem:ac-disc} we see that
  \[
  \Leg{\ac(\Disc(P_{4k+1,a}))} = \left( \frac{-a}{p}\right),
  \]
  for all primes $p$ in the range $\frac{4k+1}{2}+a<p<4k+1+a$. As $-a\notin \Q^{\times 2}$, we can use quadratic reciprocity to find that $\Leg{\ac(\Disc(P_{4k+1,a}))} = -1$ for all $p\equiv c \pmod D$ for some coprime integers $c,D$. By the prime number theorem for arithmetic progressions \cite[Chs.~20--21]{davenport2000}, the interval $\bigl(\frac{4k+1}{2}+a,4k+1+a\bigr)$ contains a prime $p\equiv c\pmod D$ for every sufficiently large $k$. Hence, from Remark \ref{rem:square-ac}, $\Disc(P_{4k+1,a})\notin \Q^{\times 2}$ for sufficiently large $k$, making $G_{4k+1,a} = S_{4k+1}$ for a density-one set of $k$.
\end{proof}
For the case of even $a>0$ and $n\equiv 0 \pmod 4$, we must use heavier machinery.

\subsection{\texorpdfstring{The Case $n\equiv 0\pmod 4$}{The Case n Congruent to 0 Modulo 4}}\label{subseq:disc-n0mod4}
For this section, assume $a>0$ is even, and that $n\equiv 0 \pmod 4$ and $n\ge a$.
In Lemma \ref{lem:ac-disc} we see that
\[
\Leg{\ac(\Disc(P_{n,a}))} = \Leg{ar!} = \Leg{a(n+a-p)!}.
\]
Note that the top parameter here depends nontrivially on $n$ and $p$, so quadratic reciprocity is not sufficient.
\begin{definition}
  For a fixed even $a>0$, we will call an integer $n$ \emph{bad} if $n\equiv 0 \pmod 4$, and for all primes $p$ in the range $\frac{n}{2}+a<p<n+a$, we have that $\Leg{a(n+a-p)!} = 1$. Denote the set of all such bad integers by $B$.
\end{definition}
Clearly, from Remark \ref{rem:square-ac}, if $\Disc(P_{n,a})\in \Q^{\times 2}$ then $n$ is bad. We will show that the set of bad values of $n$ cannot have positive upper density.
\begin{definition}
  For a set $S\subseteq\N$, its \emph{upper density}, denoted by $\ud(S)$, is defined by
  \[
  \ud(S) = \limsup_{k\to \infty}\frac{\abs{S\cap[k]}}{k}.
  \]
  Note that $S$ has density 0 if and only if $\ud(S) = 0$. 
\end{definition}
\begin{lemma}\label{lem:dense-bad-pairs}
  If $B$ has positive upper density, then there are an integer $h>1$ and a constant $c>0$ such that, for arbitrarily large $X$, the set of pairs of bad integers $n,n+h$,
  \[
  B_h \coloneqq B\cap (B-h),
  \]
  satisfies:
  \[
  \abs{B_h\cap[X,2X]}\ge cX.
  \]
\end{lemma}
\begin{proof}
  Choose $\eta>0$ such that $\abs{B\cap[X,2X]}\ge\eta X$ for arbitrarily large $X$; such an $\eta$ exists by decomposing $[1,Y]$ into dyadic intervals. For one of these $X$, put $A=B\cap[X,2X]$, choose an integer $H>4/\eta^2$, and partition $[X,2X]$ into intervals of length $H$. If $m_j$ denotes the number of elements of $A$ in the $j$th interval, Cauchy--Schwarz gives
  \[
    \sum_j\binom{m_j}{2}
    =\frac12\left(\sum_jm_j^2-\abs A\right)\gg_\eta X.
  \]
  Every pair counted on the left has difference in $\{1,\ldots,H-1\}$. Hence, for each such $X$, some $h\in\{1,\ldots,H-1\}$ satisfies $\abs{B_h\cap[X,2X]}\ge cX$, where $c>0$ depends only on $\eta$ and $H$. Passing to a subsequence makes $h$ fixed. Since every element of $B$ is divisible by $4$, this $h$ is in fact a positive multiple of $4$.
\end{proof}
These bad pairs allow us to get simpler Legendre symbols.
\begin{proposition}\label{prop:pairs-char}
  Denote by $(x)_k \coloneqq x(x+1)\cdots (x+k-1)$ the rising factorial/Pochhammer symbol. Moreover, for a fixed $a$, consider the primitive Dirichlet character $\chi_{n,h}$ associated with the extension $\Q(\sqrt{(n+a+1)_h})$. Then, a bad pair $n,n+h$ (meaning $n\in B_h$) implies that the Dirichlet character satisfies $\chi_{n,h}(p) = 1$ for all primes $p$ with $\frac{n+h}{2}+a<p<n+a$.
\end{proposition}
\begin{proof}
  If both $n,n+h$ are bad, then for any prime $p$ in the range $\frac{n+h}{2}+a<p<n+a$, one has that
  \[
  \Leg{a(n+a-p)!} = \Leg{a(n+h+a-p)!} = 1.
  \]
  Taking the ratio of the top parameters gives:
  \[
  \begin{aligned}
  \frac{a(n+h+a-p)!}{a(n+a-p)!}
    &= \frac{(n+h+a-p)!}{(n+a-p)!}\\
    &= (n+a-p+1)_h\equiv (n+a+1)_h \pmod p.
  \end{aligned}
  \]
  This means that
  \[
  \Leg{(n+a+1)_h} = 1,
  \]
  which agrees with $\chi_{n,h}(p)$, as desired.
\end{proof}
This fact is not very informative if $(n+a+1)_h\in\Q^{\times 2}$, as then $\chi_{n,h} = 1$. However, this is ruled out by a standard theorem of Erd\H{o}s and Selfridge \cite[Theorem~1, p.~292]{erdos-selfridge1975}:
\begin{theorem}
  The product of two or more consecutive positive integers is never a perfect power.
\end{theorem}
\begin{corollary}\label{cor:chi-conductor}
  If we denote the \emph{squarefree} part of an integer $x$ as $\sqf(x)$, then, for $h>1$,
  \[
  1<\sqf((n+a+1)_h)\le(n+a+1)_h.
  \]
  This implies that $\chi_{n,h}$ is a primitive nonprincipal Dirichlet character, with conductor $q_{n,h}\le 4(n+a+1)_h$.
\end{corollary}
Note this implies a very problematic scenario - each bad pair $n,n+h$ with sufficiently large $n$ gives a very biased character. This is formalized in the following lemma:
\begin{lemma}\label{lem:pairs-zero-force}
  Fix $h>1$ and $0<\delta<\frac{1}{2}$. There exists $N_0>0$ which depends on $a,h,\delta$, such that for every $n\in B_h$ with $n>N_0$, the $L$-function $L(s,\chi_{n,h})$ has a zero $\rho = \beta+i\gamma$ satisfying:
  \[
  \beta\ge 1-\delta,\qquad \abs{\gamma}\le n^{\delta}.
  \]
\end{lemma}
Let us first show how disallowing a zero bounds the character sum. We will use Vinogradov's symbol $\ll$ to denote an inequality up to a global constant, and $\ll_x$ to denote an inequality up to a constant depending on $x$.
\begin{proposition}\label{prop:zerofree-bound}
 Let $0<\delta<\frac{1}{2}$, let $q,x\ge2$, and let $\chi$ be a nonprincipal Dirichlet character mod $q$. If
  \[
  L(s,\chi)\neq 0 \qquad (\Re s>1-\delta, \abs{\Im s}\le x^{\delta}),
  \]
then
\[
\frac{1}{\pi(x)}\sum_{p\le x}\chi(p)\ll\frac{\log^3(qx)}{x^\delta}.
\]
\end{proposition}
\begin{proof}
  This is established in \cite[Proposition~3.9, p.~14]{pilatte2026}.
\end{proof}
However, this bound is contradicted by sufficiently large bad pairs, proving the existence of these zeros:
\begin{proof}[Proof of Lemma~\ref{lem:pairs-zero-force}]
  Fix $h>1$ and $0<\delta<\frac{1}{2}$.
  By Proposition \ref{prop:pairs-char}, for any sufficiently large $n\in B_h$,
  \[
  \sum_{\frac{2n}{3}<p<n}\chi_{n,h}(p) = \pi(n)-\pi\left(\frac{2n}{3}\right) \sim \frac{n}{3\log(n)}.
  \]
  However, if $L(s,\chi_{n,h})$ has no zero $\rho = \beta+i\gamma$ satisfying:
  \[
  \beta\ge 1-\delta,\qquad \abs{\gamma}\le n^{\delta},
  \]
  then Proposition \ref{prop:zerofree-bound} gives
  \[
  \begin{aligned}
  \sum_{\frac{2n}{3}<p<n}\chi_{n,h}(p)
  &\ll \frac{\log^3(q_{n,h}n)}{n^\delta}\pi(n)
  +\frac{\log^3\!\left(q_{n,h}\frac{2n}{3}\right)}{\left(\frac{2n}{3}\right)^\delta}
   \pi\!\left(\frac{2n}{3}\right)\\
  &\ll_{\delta} n^{1-\delta}\frac{\log^3(q_{n,h}n)}{\log(n)}.
  \end{aligned}
  \]
  By Corollary \ref{cor:chi-conductor}, $q_{n,h}<4(n+a+1)_h\ll_{a,h}n^h$. Thus,
  \[
  \sum_{\frac{2n}{3}<p<n}\chi_{n,h}(p)\ll_{a,h,\delta} n^{1-\delta} \log(n)^2 = o\left(\frac{n}{\log(n)}\right).
  \]
  Thus there exists some $N_0$, depending on $a,h,\delta$, after which for all $n>N_0$ the two estimates are incompatible. This implies that for all such $n$ the $L$-function $L(s,\chi_{n,h})$ has a zero in the relevant rectangle, as desired.
\end{proof}
Now, while this implies an immediate proof which is conditional on GRH, an unconditional proof is possible, using a zero-density bound on the number of zeros of the Dirichlet $L$-functions associated with primitive characters with bounded conductor. 
\begin{definition}
  For $\alpha\in(1/2,1]$, let $N(\alpha,T,\chi)$ denote the number of zeros $\beta+i\gamma$ of $L(s,\chi)$ satisfying
  \[
  \alpha\le\beta\le 1,\qquad\abs{\gamma}\le T.
  \]
  Moreover, define
  \[
  N^*(\alpha,T,Q) \coloneqq \sum_{q\le Q} \sum_{\chi \text{ mod } q}^* N(\alpha,T,\chi).
  \]
  Here the asterisk indicates that the inner sum is over primitive characters.
\end{definition}
\begin{theorem}\label{thm:zero-dense}
  For $\alpha\in(\frac{1}{2},1]$ and $T\ge 1$, we have
\[
N^*(\alpha,T,Q)\ll_{\varepsilon}(Q^5T^3)^{(1+\varepsilon)(1-\alpha)},
\]
for any $\varepsilon>0$. 
\end{theorem}
\begin{proof}
  This is Selberg's zero-density estimate \cite[p.~216]{selberg1972}.
\end{proof}
The argument we make is essentially:
\[
\begin{aligned}
  &\text{many bad pairs } n,n+h\\
  &\quad\implies \text{ many bad characters } \chi_{n,h}\\
  &\quad\implies \text{contradiction to Theorem \ref{thm:zero-dense}}.
\end{aligned}
\]
Let us start with the first implication.
\begin{lemma}\label{lem:manypairs-manychis}
  If $\abs{B_h\cap[X,2X]}> cX$, then 
  \[
  \abs{\{\chi_{n,h}:n\in B_h\cap[X,2X] \}}\gg_{a,h}X^{\frac{1}{2}-o(1)}.
  \]
\end{lemma}
To show this, we need the following proposition:
\begin{proposition}\label{prop:little-intersections}
  \[
  \abs{\{n\in[X,2X]:\sqf((n+a+1)_h) = D\}}\ll_{a,h} X^{\frac{1}{2}+o(1)}.
  \]
\end{proposition}
\begin{proof}
  Say $n\in[X,2X]$, and $\sqf((n+a+1)_h) = D$.
  First note that any prime $p>h$ that divides $\sqf(n+a+1)$ will not divide $n+a+j$ for $1<j\le h$. Hence $p|D$. This implies:
  \[
  \sqf(n+a+1)|D\prod_{p\le h}p.
  \]
  However, the number of divisors of the right hand side is at most $2^{\omega(D)+h}$. Since $D\ll_{a,h} X^h$, we have $2^{\omega(D)+h} = X^{o(1)}$. We are left with enumerating the possibilities of $\frac{n+a+1}{\sqf(n+a+1)}$. But this is $\ll X^{\frac{1}{2}}$. Thus, the number of sources for $D$ is at most the product, which gives:
    \[
  \abs{\{n\in[X,2X]:\sqf((n+a+1)_h) = D\}}\ll_{a,h} X^{\frac{1}{2}+o(1)}.
  \]
\end{proof}
The desired lemma now follows immediately. 
\begin{proof}[Proof of Lemma~\ref{lem:manypairs-manychis}]
  As distinct squarefree parts imply distinct characters, we must simply divide the number of pairs by the maximal number of possible intersections. This is given by Proposition \ref{prop:little-intersections}, and implies the desired:
  \[
  \abs{\{\chi_{n,h}:n\in B_h\cap[X,2X] \}}\gg_{a,h}\frac{\abs{B_h\cap[X,2X]}}{X^{\frac{1}{2}+o(1)}} \gg X^{\frac{1}{2}-o(1)}.
  \]
\end{proof}
This implies many zeros in a specific rectangle:
\begin{lemma}
  Fix $h>1$ and $0<\delta<\frac{1}{2}$.
  If $\abs{B_h\cap[X,2X]}> cX$ for arbitrarily large $X$, then
  \[
  N^*(\alpha,T,Q)\gg_{a,h}X^{\frac{1}{2}-o(1)},
  \]
  for 
  \[
  \alpha = 1-\delta,\qquad T = (2X)^\delta,\qquad Q = 4(2X+a+1)_h.
  \]
\end{lemma}
\begin{proof}
  Fix $h>1$ and $0<\delta<\frac{1}{2}$, and assume $X>N_0$ from Lemma \ref{lem:pairs-zero-force}. Then, all characters in the set
  \[\{\chi_{n,h}:n\in B_h\cap[X,2X]\}\] have a zero in the relevant rectangle. Moreover, their conductors are all bounded by $Q = 4(2X+a+1)_h$. Hence, by definition of $N^*$, and Lemma \ref{lem:manypairs-manychis}, we have the desired:
   \[
  N^*(\alpha,T,Q)\gg_{a,h}X^{\frac{1}{2}-o(1)}.
  \]
\end{proof}
\begin{lemma}\label{lem:bad-density-zero}
  $B$ has density zero.
\end{lemma}
\begin{proof}
  Assume that $B$ has positive upper density. By Lemma~\ref{lem:dense-bad-pairs}, there exist an integer $h>1$ and a constant $c>0$ such that $\abs{B_h\cap[X,2X]}> cX$ for arbitrarily large $X$.
  Hence \[
  N^*(\alpha,T,Q)>cX^{\frac{1}{2}-o(1)},
  \]
  for 
  \[
  \alpha = 1-\delta,\qquad T = (2X)^\delta,\qquad Q = 4(2X+a+1)_h.
  \]
  But note that
  \[
N^*(\alpha,T,Q)\ll_{\varepsilon}(Q^5T^3)^{(1+\varepsilon)(1-\alpha)} \ll_{h,a,\varepsilon} (X^{5h}X^{3\delta})^{(1+\varepsilon)\delta} = X^{(1+\varepsilon)\delta(5h+3\delta)}.
\]
Hence, for sufficiently small $\delta$ and $\varepsilon = 1$, we will have $N^*(\alpha,T,Q) = o(X^{\frac{1}{3}})$, which contradicts the lower bound on $N^*(\alpha,T,Q)$.
\end{proof}

This \emph{finally} implies the main result:
\begin{proof}[Proof of Theorem~\ref{thm:main-target}]
  Fix $a>0$. Proposition~\ref{prop:A_n-inclusion} shows that $A_n\le G_{n,a}$ outside a set of density zero. The case of odd $a$ was proved above, so suppose that $a$ is even. If $n\equiv2,3\pmod4$, Lemma~\ref{lem:disc_sign} makes the discriminant negative outside a set of density zero. If $n\equiv1\pmod4$, Corollary~\ref{cor:even-a-1mod4} gives $G_{n,a}=S_n$ outside a set of relative density zero in that congruence class. Finally, if $n\equiv0\pmod4$, every $n$ for which $\Disc(P_{n,a})$ is a square belongs to $B$, while Lemma~\ref{lem:bad-density-zero} shows that $B$ has density zero. Thus, in every congruence class, the exceptional values of $n$ form a density-zero set. Outside their union, $A_n\le G_{n,a}$ and the discriminant is not a square, so $G_{n,a}\not\le A_n$. Therefore $G_{n,a}=S_n$, as claimed.
\end{proof}
\pagebreak

\end{document}